\documentclass[11pt]{amsart}
\usepackage{mathrsfs}
\usepackage{amsfonts}
\usepackage{amsmath}
\usepackage{amssymb}
\usepackage{amsthm}
\usepackage{enumerate}
\usepackage{color}
\usepackage{geometry}
\usepackage{hyperref}
\usepackage{float}
\usepackage{graphicx}
\usepackage{multirow}
\usepackage[numbers,sort&compress]{natbib}
\usepackage{color}
\usepackage[all]{xy}
\usepackage{cases}

\allowdisplaybreaks
\numberwithin{equation}{section}

\theoremstyle{plain}
\newtheorem{prop}{Proposition}[section]
\newtheorem{coro}[prop]{Corollary}

\newtheorem{lemm}[prop]{Lemma}

\newtheorem{theorem}[prop]{Theorem}

\theoremstyle{definition}

\renewcommand\aa{a}

\newcommand\bb{b}

\newcommand\ff{f}
\renewcommand\gg{g}

\newcounter{ITEM}
\newcommand\ITEM[1]{\setcounter{ITEM}{#1}\leavevmode\hbox{\rm(\roman{ITEM})}}

\newcommand\mA{\mathcal{A}}

\newcommand\mP{\mathcal{P}}

\newcommand\uu{u}

\newcommand\xx{x}
\renewcommand\xi[1]{x_{i_{#1}}}

\newcommand\yy{y}
\newcommand\zz{z}

\title{An Anick type wild automorphism of free Poisson algebras}

\author{Ivan Shestakov$^{\ddag}$}
\address{I.S., Instituto de Matem\'{a}tica e Estat\'{i}stica, Universidade de S\~{a}o
Paulo, Brazil; Sobolev Institute of Mathematics, Novosibirsk, Russia.}
\email{shestak@ime.usp.br}

\author{Zerui Zhang$^*$}
\address{Z.Z., School of Mathematical Sciences, South China Normal University, Guangzhou 510631, P. R. China}
\email{\small zeruizhang@scnu.edu.cn}

\thanks{${}^{\ddag}$ Supported by grants FAPESP 2018/23690-6, and CNPq  304313/2019-0}
\thanks{${}^*$ Supported by the NNSF of China (Grant No.12101248) and by the
Guangdong Basic and Applied Basic Research Foundation (Grant No. 2024A1515013122)}
\thanks{${}^*$  Corresponding author}


\begin{document}

\begin{abstract}

We construct an Anick type wild automorphism ~$\delta$ in a 3-generated free Poisson algebra which  induces a tame automorphism in a 3-generated polynomial algebra.  We also show that~$\delta$ is stably tame.  \\
\begin{center}
  Dedicated to the memory of professor V.A.Roman'kov
\end{center}
\end{abstract}

\maketitle

\section{Introduction}\label{Intro}
We recall that a \emph{Poisson algebra} is a vector space~$\mP$ over a field~$k$ endowed with two bilinear operations:  a multiplication denoted by~$\xx\otimes \yy\mapsto \xx\cdot \yy$ and a Poisson bracket denoted by~$\xx\otimes \yy\mapsto [\xx,\yy]$,  such that~$(\mP, \cdot)$ is a commutative associative algebra, $(\mP, [-,-])$ is a Lie algebra, and~$\mP$ satisfies the \emph{Leibniz identity}
$$[\xx, \yy \cdot \zz]= \yy \cdot [\xx, \zz]+ \zz\cdot[\xx, \yy].$$
A Poisson algebra~$\mP$ is called \emph{unital} if $(\mP,\cdot)$ is unital. In this article, we always assume that the Poisson algebra under consideration is unital.
And we usually omit the operation~$\cdot$ when no confusion arises.

It is clear that polynomial algebras are Poisson algebras with trivial Poisson brackets. On the other hand, assume that $\mathcal{L}$ is the free Lie algebra generated by an arbitrary set $X$ and assume that~$Y$ is a linear basis of~$\mathcal{L}$. Then the Poisson symmetric algebra of~$\mathcal{L}$, namely, the polynomial algebra with free generating set~$Y$, is the free Poisson algebra generated by~$X$~\cite{Sh93}. Since research on automorphisms of polynomial algebras is an interesting subject, it is natural to study automorphisms of free Poisson algebras.

Let us now review some essential terminology. Let~$A_n$ be an arbitrary free algebra with free generating set~$X:=\{x_1,\dots, x_n\}$ and denote by~$\mathsf{Aut}(A_n)$ the automorphism group of~$A_n$. We denote an automorphism of the form
$$\sigma(i,\alpha,f):=(x_1,\dots, x_{i-1}, \alpha x_i+f, x_{i+1},\dots, x_n),$$
where~$1\leq i\leq n$, $0\neq \alpha\in k$, and~$f$ lies in the subalgebra of~$A_n$ generated by~$X\setminus\{x_i\}$,
to be an \emph{elementary automorphism}  of~$A_n$.
The subgroup~$\mathsf{TA}(A_n)$ of~$\mathsf{Aut}(A_n)$ generated by all elementary automorphisms is known as the \emph{tame automorphism group} of~$A_n$.  We refer to the
elements of~$\mathsf{TA}(A_n)$ as \emph{tame automorphisms} and those in~$\mathsf{Aut}(A_n)\setminus\mathsf{TA}(A_n)$ as \emph{wild automorphisms}. We denote by
$$\varphi=(f_1, \dots, f_n)$$
the endomorphism~$\varphi$ of~$A_n$ defined by the rule $\varphi(x_i)=f_i\in A_n$ for all~$i\leq n$.

Known results have demonstrated close connections between the automorphisms of polynomial algebras,  free associative algebras, and free Poisson algebras.
For instance, automorphisms of polynomial algebras (free associative algebras and free Poisson algebras, respectively) in two variables are tame~\cite{ass-rank,com-tame,com-tam,ass-tame,pos-tame}, where the characteristic of the underlying fields is assumed to be 0 for Poisson algebras.  Moreover,  the automorphism groups of the aforementioned two-generated free algebras are isomorphic.

On the other hand, automorphisms of a polynomial algebra and of a free  associative algebras in three variables can be wild,  and the most famous example are the Nagata automorphism \cite{Nagata, US041,US042}  and the Anick automorphism \cite{Umir07Anick}.  Since every automorphism of an $n$-generated polynomial algebra can be lifted to an automorphism of an $n$-generated free Poisson algebra, it follows that $3$-generated free Poisson algebras have wild automorphisms  as well.  However, to our knowledge, it remains an open question whether every wild automorphism of an $n$-generated free Poisson algebra induces a wild automorphism of an $n$-generated polynomial algebra under the projection sending Poisson brackets to zero. In this article we construct an Anick type wild automorphism
\begin{equation}\label{Anick-auto}
\delta: =(x_1+[x_3, x_1x_3-[x_3,x_2]], x_2+(x_1x_3-[x_3,x_2])x_3, x_3)
\end{equation}
of a 3-generated free Poisson algebra, which induces a tame automorphism of a polynomial algebra in 3 variables. Since~$\delta(x_1x_3-[x_3,x_2])=x_1x_3-[x_3,x_2]$, we deduce that~$\delta$ is an epimorphism. By a similar proof as that in \cite[p. 343, Proposition 8.1]{Cohn},  we conclude that~$\delta$ is an automorphism.

The article is organized as follows: In Section~\ref{SS-TA}, we explore the connections between the group of tame automorphisms of a free Poisson algebras and that of a polynomial algebra. In Section~\ref{SS-JM}, we define Jacobian matrices for endomorphisms of free Poisson algebras and prove the chain rule. Although the chain rule holds in general~\cite{Dren93obstruction}, for the sake of completeness and for the readers' convenience, we still include a proof. In Section~\ref{SS-wild},  we establish our main result: the automorphism~$\delta$ in~\eqref{Anick-auto} is wild and stably tame. 

Our approach is primarily inspired by~\cite{Umir07Anick}, with modifications to make the proofs simplified.

\section{Tame automorphisms of free Poisson algebras}\label{SS-TA}
Our aim in this section is to  investigate the connections between tame automorphisms of free Poisson algebras and those of polynomial algebras.  Denote by~$P_n$ the free Poisson algebra with free generating set~$X:=\{x_1,\dots, x_n\}$. Let~$C_n$ be the polynomial algebra with variables~$x_1,\dots, x_n$, and we always suppose that~$C_n$ is endowed with the trivial Poisson brackets. Then there exists a Poisson algebra homomorphism
$$\pi_n: P_n \longrightarrow C_n, \quad  x_i\mapsto x_i, \mbox{ for all } 1\leq i\leq n.$$
The kernel~$\mathsf{Ker}(\pi_n)$ of~$\pi_n$ is the ideal of~$P_n$ generated by~$\{[x_i, x_j]\mid x_i, x_j\in X, i>j\}$. Finally, for every endomorphism~$\varphi=(f_1,\dots, f_n)$ of~$P_n$, we denote by~$\overline{\varphi}=(\pi_n(f_1),\dots, \pi_n(f_n))$ the induced endomorphism of~$\varphi$ in~$C_n$. If~$\varphi$ is an automorphism, then so is~$\overline{\varphi}$.
Obviously, every elementary automorphism~$\varphi$ of the form
\begin{equation}\label{AutInKer}
\sigma(i,1,f), f\in \mathsf{Ker}(\pi_n),
\end{equation}
satisfies that~$\overline{\varphi}=\mathsf{id}$, the identity map of the algebra~$C_n$.

For~$p\neq q$, we denote by~$\sigma_{pq}$ the automorphism of~$A_n$ that interchanges~$x_p$ and~$x_q$,  while the remaining generators are fixed. Then~$\sigma_{pq}$ is tame and we have
$$\sigma_{pq}=\sigma(q,-1,x_p)\sigma(p,1,-x_q)\sigma(q,1,x_p).$$
It is straightforward to see that the following relations~\cite{Umir06Relation,Umir07Anick} between elementary automorphisms hold for~$A_n$:
\begin{equation}\label{relation1}
\sigma(i,\alpha,f)\sigma(i,\beta,g)=\sigma(i,\alpha\beta, \beta f+g);
\end{equation}
\begin{equation}\label{relation2}
\sigma(i,\alpha,f)^{-1}\sigma(j,\beta,g)\sigma(i,\alpha,f)=\sigma(j,\beta,\sigma(i,\alpha,f)^{-1}(g)),
\end{equation}
where~$i\neq j$ and~$f$ lies in the subalgebra of~$A_n$ generated by~$X\setminus\{x_i,x_j\}$;
\begin{equation}\label{relation3}
 \sigma_{pq}\sigma(i,\alpha,f) \sigma_{pq}=\sigma(j,\alpha,\sigma_{pq}(f)),
\end{equation}
where~$1\leq p\neq q\leq n$ and~$x_j=\sigma_{pq}(x_i)$.

For~$j>1$, by straightforward calculations or by~\cite{Umir07Anick}, we know that~$\sigma_{1j}$ satisfies the property
 \begin{equation}\label{sigma1j}
 \sigma_{1j}=\sigma(1,-1,0)\sigma(1,1,x_j)\sigma(j,1,-x_1)\sigma(1,1,x_j).
 \end{equation}
And thus, by~\eqref{relation3}, we obtain
\begin{equation}\label{sigma2j}
   \sigma(j,\alpha, g)=\sigma_{1j}\sigma(1,\alpha, \sigma_{1j}(g))\sigma_{1j}.
\end{equation}

Note that~\eqref{sigma1j} and~\eqref{sigma2j} are general facts involving only elementary automorphisms and do not depend on the underlying algebra.

In general, the relations \eqref{relation1}-\eqref{relation3} do not form a complete set of defining relations of~$\mathsf{TA}(A_n)$, but they do if~$A_n=C_3$ is the polynomial algebra with variables~$x_1,x_2,x_3$ over a field of characteristic 0~\cite{Umir06Relation}.

\begin{lemm}\label{kernel-aut}
Let~$\Pi$ be the group homomorphism as follows:
  $$\Pi : \mathsf{TA}(P_n)\longrightarrow \mathsf{TA}(C_n),  \ \varphi\mapsto \overline{\varphi}.$$
  Then the kernel of~$\Pi$, as a normal subgroup of $\mathsf{TA}(P_n)$, is generated by all elementary automorphisms of the form given in~\eqref{AutInKer}.
\end{lemm}

\begin{proof}
   Let~$N$ be the normal subgroup of $\mathsf{TA}(P_n)$ generated by all elementary automorphisms of the form given in~\eqref{AutInKer}. Clearly we have~$N\subseteq \mathsf{Ker}(\Pi)$, so we obtain a well-defined group homomorphism
$$\Pi_1: \mathsf{TA}(P_n)/N \longrightarrow \mathsf{TA}(C_n),  \ \varphi N\mapsto \overline{\varphi}.$$
It suffices to show that~$\Pi_1$ is an isomorphism.

For every element~$f\in P_n$, we have $f=f_0+f_1$, where~$f_0$ does not involve Poisson brackets and $f_1$ lies in~$\mathsf{Ker}(\pi_n)$.
Moreover, for~$\sigma(i,\alpha, f)\in \mathsf{TA}(P_n)$, by~\eqref{relation1}, we have
$$\sigma(i,\alpha, f)=\sigma(i,\alpha, f_0)\sigma(i,1, f_1).$$
So we deduce~$\sigma(i,\alpha, f)N=\sigma(i,\alpha, f_0)N$ and
\begin{equation}\label{pisigma}
 \Pi_1(\sigma(i,\alpha, f)N)=\Pi_1(\sigma(i,\alpha, f_0)N)
=\overline{\sigma(i,\alpha,f_0)}=\sigma(i,\alpha,f_0)\in \mathsf{TA}(C_n).
\end{equation}
On the other hand,  let~$\varphi=(g_1,\dots, g_n)$ be a (tame) automorhphism of~$C_n$. Considering $g_1,\dots,g_n$ as elements in~$P_n$, we know that $\varphi_{_P}:=(g_1,\dots,g_n)\in \mathsf{Aut}(P_n)$.   Define
 a map
$$\Omega : \mathsf{TA}(C_n) \longrightarrow \mathsf{TA}(P_n)/N,\  \varphi\mapsto \varphi_{_P}N.$$
Clearly, $\Omega$ is a group homomorphism and
 we have
 \begin{equation}\label{omesigma}
   \Omega( \sigma(i,\alpha,f))=\sigma(i,\alpha,f)N.
 \end{equation}
 Since elementary automorphisms generate the tame automorphism groups, by \eqref{pisigma} and~\eqref{omesigma}, we have
$$\Omega \Pi_1=\mathsf{Id}  \mbox{ and } \Pi_1\Omega=\mathsf{Id'},$$
where $\mathsf{Id} $ is the identity map of $\mathsf{TA}(P_n)/N$ and~$\mathsf{Id'}$ is the identity map of~$\mathsf{TA}(C_n) $. The result follows immediately.
 \end{proof}

\section{Jacobian matrices and the chain rule}\label{SS-JM}
In this section, we first recall from~\cite{Umir12pe} the universal (multiplicative) enveloping algebra~$\mP^e$ of an arbitrary Poisson algebra~$\mP$ in the variety of all Poisson algebras. Then we define the Jacobian matrices for endomorphisms of $P_n$. Finally, we prove that the chain rule holds.

Let~$M_{\mP}=\{M_a\mid a\in \mP\}$ and let~$H_{\mP}=\{H_{a}\mid a\in \mP\}$ be two copies of~$\mP$ such that~$M: \mP\longrightarrow M_{\mP}$ sending~$a$ to~$M_a$ and $H: P\longrightarrow H_{\mP}$ sending~$a$ to~$H_a$  are linear isomorphisms. Then $\mP^e$ is the associative algebra generated by the spaces~$M_{\mP}$ and~$H_{\mP}$, with defining relations
\begin{equation}\label{rela-pe1}
  M_{a}M_{b}=M_{ab},
\end{equation}
\begin{equation}\label{rela-pe2}
H_{[a,b]}=H_aH_b-H_bH_a,
\end{equation}
\begin{equation}\label{rela-pe3}
H_{ab}=M_bH_a+M_aH_b,
\end{equation}
\begin{equation}\label{rela-pe4}
H_aM_b=M_{[a,b]}+M_bH_a,
\end{equation}
\begin{equation}\label{rela-pe5}
M_1=1
\end{equation}
for all~$a,b\in \mP$. It follows from~\eqref{rela-pe3} and~\eqref{rela-pe5} that~$H_1=0$.
Moreover, every (unital) $\mP$-module~$\mathcal{V}$ is a left $\mP^e$-module with the actions
$$M_a(v)=a\cdot v, \mbox{ and } H_{a}(v)=[a,v]$$ for all~$a\in \mP$ and~$v\in \mathcal{V}$.  Assume that~$\mP$ is generated by~$X$. Rewriting all the terms on the left-hand side of the above equations in~\eqref{rela-pe1}-\eqref{rela-pe4}, we observe that elements of the form~$M_aH_{x_{i_1}}...H_{x_{i_t}}$ with~$t\geq 0$ forms a linear generating set of~$\mP^e$.
Define
$$P_n^e=(P_n)^e \mbox{ and } C_n^e=(C_n)^e.$$
Then we have
\begin{lemm}\cite[Corollary 3]{Umir12pe}\label{basis}
  Every nonzero element in~$P_n^e$ can be uniquely written in the form
  $$\sum_{1\leq i\leq q}M_{a_i}w_i,$$ where each~$a_i$ lies in~$P_n$, each~$w_i$ is of the form~$H_{x_{i_1}}...H_{x_{i_t}}$, where~$t$ depends on~$i$,   and each~$w_i$ is distinct from the others.
\end{lemm}

Since~$C_n$ is a Poisson algebra with trivial Poisson brackets, by Equations~\eqref{rela-pe1}, \eqref{rela-pe2} and~\eqref{rela-pe4}, we deduce that~$C_n^e$ is a commutative algebra. Combining this with Lemma~\ref{basis}, we obtain  the following corollary:
\begin{coro}\label{c1e}
$C_1^e=P_1^e$ is the polynomial algebra with variables~$M_{x_1}$ and~$H_{x_1}$.
  \end{coro}

We further observe that homomorphisms of Poisson algebras induce homomorphisms of their corresponding universal enveloping algebras. Let~$\mP_1$ and $\mP_2$ be two Poisson algebras, and let~$\varphi$ be a Poisson algebra homomorphism from~$\mP_1$ to~$\mP_2$. Since~$(\mP_1)^e$ and~$(\mP_2)^e$ are both defined by relations of the form \eqref{rela-pe1}-\eqref{rela-pe5}, we can construct an associative algebra homomorphism
$$\varphi^e: (\mP_1)^e\longrightarrow (\mP_2)^e,\ M_a\mapsto M_{\varphi(a)},\ H_a\mapsto H_{\varphi(a)}$$
for all~$a\in \mP_1$. In particular, we obtain an algebra homorphism
$$\pi_n^e: P_n^e \longrightarrow C_n^e.$$

Now we recall that the Fox derivatives $\frac{\partial(-)}{\partial x_i}$  is a linear map from~$P_n$ to~$P_n^e$  defined by the rule:
$$\frac{\partial\xx_j}{\partial x_i} =\delta_{i,j},$$
$$\frac{\partial(ab)}{\partial x_i}=M_a\frac{\partial b}{\partial x_i}+M_b\frac{\partial a}{\partial x_i}, $$
$$\frac{\partial[a,b]}{\partial x_i}=H_a\frac{\partial b}{\partial x_i}-H_b\frac{\partial a}{\partial x_i}, $$
for all~$a,b\in P_n$, where~$\delta_{i,i}=1$ and $\delta_{i,j}=0$ if $i\neq j$. It is straightforward to show that~$\frac{\partial(-)}{\partial x_i}$ is well-defined.

Let~$\mP$ be a Poisson algebra. For all~$a_1,\dots, a_t \in \mP$, define
$$[a_1, \dots, a_t] =[\dots[[a_1,a_2], a_3], \dots, a_t],$$
where~$[a_1, \dots, a_t]$ means~$a_1$ if~$t=1$. Then we have the following formula for the Fox derivative acting on elements involving Poisson brackets.

\begin{lemm}\label{par-at}
Let~$\mP$ be a Poisson algebra. Then for all~$a_1,\dots a_t\in \mP$ with~$t\geq 2$, we have
$$\frac{\partial{[a_1, \dots, a_t]}}{\partial\xx_r}
=\sum_{1\leq j\leq t}(-H_{a_t})\dots(-H_{a_{j+1}})H_{[a_1,\dots,a_{j-1}]}\frac{\partial\aa_j}{\partial\xx_r}.$$
\end{lemm}

\begin{proof}
We use induction on~$t$. For~$t=2$,  we have
$$\frac{\partial{[a_1, a_2]}}{\partial\xx_r}
=(-H_{a_2})\frac{\partial\aa_1}{\partial\xx_r}+H_{a_1}\frac{\partial\aa_2}{\partial\xx_r}.$$
Now we assume that~$t>2$. We define~$b_1=[a_1, a_2]$, and $b_j=a_{j+1}$ for all~$2\leq j\leq t-1$. By induction hypothesis, we have
\begin{align*}
&\frac{\partial{[a_1, \dots, a_t]}}{\partial\xx_r}&\\
=&\sum_{1\leq p\leq t-1}
(-H_{b_{t-1}})\dots(-H_{b_{p+1}})H_{[b_1,\dots,b_{p-1}]}\frac{\partial\bb_p}{\partial\xx_r}&\\
=&\sum_{2\leq p\leq t-1}(-H_{a_t})\dots(-H_{a_{p+2}})H_{[a_1,\dots,a_p]}\frac{\partial\aa_{p+1}}{\partial\xx_r}
+(-H_{a_t})\dots(-H_{a_3})\frac{\partial[a_1,a_2]}{\partial\xx_r}&\\
=&\sum_{3\leq j\leq t} (-H_{a_t})\dots(-H_{a_{j+1}})H_{[a_1,\dots,a_{j-1}]}\frac{\partial\aa_{j}}{\partial\xx_r}
 +(-H_{a_t})\dots(-H_{a_3})
((-H_{a_2})\frac{\partial\aa_1}{\partial\xx_r}+H_{a_1}\frac{\partial\aa_2}{\partial\xx_r}).&
\end{align*}
The proof is completed.
\end{proof}

Now we define the Jacobian matrices for endomorphisms of~$P_n$ in a way that differs slightly from~\cite{Umir12pe}, as it will be more convenient to deal with matrices over a commutative associative algebra rather than a noncommutative one.
For an arbitrary associative algebra~$\mA$, we denote by~$M_{n\times n}(\mA)$ the set of all $n\times n$ matrices over~$\mA$. For every endomorphism~$\varphi$ of~$P_n$, we define the Jacobian matrix~$J(\varphi)$ of~$\varphi$ to be
$$J(\varphi)=\pi_n^e((\frac{\partial\varphi(x_i)}{\partial\xx_j})_{n\times n})
=(\pi_n^e(\frac{\partial\varphi(x_i)}{\partial\xx_j}))_{n\times n}\in M_{n\times n}(C_n^e),$$
where~$\pi_n^e(\frac{\partial\varphi(x_i)}{\partial\xx_j})$ lies in the $i$-th row and~$j$-th column of~$J(\varphi)$. Since~$C_n^e$ is commutative, the determinant~$|J(\varphi)|$ of~$J(\varphi)$ makes sense.

Recall that every endomorphism~$\varphi$ of~$P_n$ induces an endomorphism~$\overline{\varphi}$ of~$C_n$ satisfying
$$\overline{\varphi}\pi_n=\pi_n\varphi,$$
namely, the diagram
$$
\xymatrix{
P_n \ar[d]_{\pi_n} \ar[r]^{\varphi}
                &  P_n \ar[d]^{\pi_n}  \\
C_n  \ar[r]_{\overline{\varphi}}
                & C_n             }
$$
  is commutative.
In particular, it follows immediately that
\begin{equation}\label{indu-e}
  \overline{\varphi}^e\pi_n^e=(\overline{\varphi}\pi_n)^e=(\pi_n\varphi)^e=\pi_n^e\varphi^e.
\end{equation}
We are now ready to study the form of the chain rule under the aforementioned definition for Jacobian matrices, the proof of which is essentially the same as that in~\cite{Dren93obstruction}.

\begin{lemm}(chain rule)\label{chain-rule}
Let~$\varphi$ and~$\psi$ be two endomorphisms of~$P_n$. Then we have
$$J(\varphi\psi)=\overline{\varphi}^e(J(\psi))J(\varphi).$$
In particular, if~$\psi$ is an automorphism, then~$J(\psi)$ is invertible and~$J(\psi)^{-1}=\overline{\psi}^e(J(\psi^{-1}))$.
 \end{lemm}
\begin{proof}
Assume that~$\psi(x_i)=f=f(x_1,\dots, x_n)$. Then we have~$\varphi\psi(x_i)=f(\varphi(x_1),\dots, \varphi(x_n))$. It follows that
\begin{align*}
  \pi_n^e(\frac{\partial\varphi\psi(x_i)}{\partial\xx_j})
&=\sum_{1\leq t\leq n}\pi_n^e\varphi^e(\frac{\partial\ff}{\partial\xx_t})\pi_n^e(\frac{\partial{\varphi(x_t)}}{\partial\xx_j}) &\\
&=\sum_{1\leq t\leq n}  \overline{\varphi}^e\pi_n^e(\frac{\partial\ff}{\partial\xx_t})\pi_n^e(\frac{\partial{\varphi(x_t)}}{\partial\xx_j}) &\\
&= \overline{\varphi}^e(\pi_n^e(\frac{\partial\psi(x_i)}{\partial\xx_1}),\dots, \pi_n^e(\frac{\partial\psi(x_i)}{\partial\xx_n}))(\pi_n^e(\frac{\partial{\varphi(x_1)}}{\partial\xx_j}),\dots, \pi_n^e(\frac{\partial{\varphi(x_n)}}{\partial\xx_j}))^T,&
\end{align*}
where~$(-)^T$ means the transposed matrix of the underlying matrix.
It follows that~$J(\varphi\psi)=\overline{\varphi}^e(J(\psi))J(\varphi)$.
Consequently, we obtain
$$E_{n}=J(\psi^{-1}\psi)=\overline{\psi^{-1}}^e(J(\psi))J(\psi^{-1})$$
and
$$E_{n}=J(\psi\psi^{-1})=\overline{\psi}^e(J(\psi^{-1}))J(\psi),$$
where~$E_n$ is the identity matrix.
So we have
$$E_{n}=\overline{\psi}^e(E_n)
=\overline{\psi}^e\overline{\psi^{-1}}^e(J(\psi))\overline{\psi}^e(J(\psi^{-1}))
=J(\psi)\overline{\psi}^e(J(\psi^{-1})).$$
The result follows immediately.
\end{proof}

Now we can deduce more properties of Jacobian matrices by applying the chain rule.
\begin{lemm}\label{app-chain-rule}
Let~$\varphi$ be an endomorphism of~$P_n$ such that~$\overline{\varphi}=\mathsf{id}$. Then for every automorphism~$\psi\in \mathsf{Aut}(P_n)$, we have
$$J(\psi\varphi\psi^{-1})=J(\psi)^{-1}\overline{\psi}^e(J(\varphi))J(\psi)$$
and
$$|J(\psi\varphi\psi^{-1})|= \overline{\psi}^e(|J(\varphi)|). $$
\end{lemm}
\begin{proof}
  By the chain rule, we have
  $$J(\psi\varphi\psi^{-1})=\overline{\psi}^e(J(\varphi\psi^{-1}))J(\psi)
  =\overline{\psi}^e(\overline{\varphi}^e(J(\psi^{-1}))J(\varphi))J(\psi)
  =J(\psi)^{-1}\overline{\psi}^e(J(\varphi))J(\psi).$$
  The other assertion follows immediately since~$\overline{\psi}^e$ is an algebra homomorphism.
\end{proof}

Let~$P\{x_3\}$ be the free Poisson algebra generated by~$x_3$. Then~$P\{x_3\}$ is the polynomial algebra endowed with trivial Poisson brackets. In order to prove that the Anick type automorphism~$\delta$ in \eqref{Anick-auto} is wild, we shall need the following  Poisson algebra homomorphism:
$$\eta: C_3\longrightarrow P\{x_3\}, x_1\mapsto 0, \ x_2\mapsto 0, \ x_3\mapsto x_3.$$
Clearly, $\eta$ is well-defined. So for every endomorphism~$\varphi$ of~$P_3$, we have
$$\eta^e(J(\varphi))=(\eta^e\pi_3^e(\frac{\partial\varphi(x_i)}{\partial\xx_j}))_{n\times n}.$$

We shall need the following observation about Fox derivatives acting on elements of~$\mathsf{Ker}(\pi_3)$ with respect to~$x_3$.
\begin{lemm} \label{par-ker}
  For every~$f\in \mathsf{Ker}(\pi_3)$, we have~$\eta^e\pi_3^e(\frac{\partial{f}}{\partial\xx_3})=0$.
\end{lemm}
\begin{proof}
For all~$x_{i_1},\dots, x_{i_t}\in X$ such that~$x_{i_2}<x_{i_1}\leq 3$, we have~$\eta\pi_3(x_{i_2})=0=\frac{\partial{x_{i_2}}}{\partial\xx_3}$.  And thus by Lemma~\ref{par-at}, we deduce
that
$$\eta^{e}\pi_3^e(\frac{\partial{[x_{i_1},\dots, x_{i_t}]}}{\partial\xx_3})=
\sum_{1\leq j\leq t}(-H_{\eta\pi_3({x_{i_t}})})
\dots(-H_{\eta\pi_3({x_{i_{j+1}}})})H_{\eta\pi_3([x_{i_1},\dots,x_{i_{j-1}}])}
\eta^e\pi_3^e(\frac{\partial\xx_{i_j}}{\partial\xx_3})=0.$$
Therefore, for every element of the form~$[x_{i_1},\dots, x_{i_t}]g\in \mathsf{Ker}(\pi_3)$ ($t\geq 2$), we have
 $$\eta^e\pi_3^e(\frac{\partial([x_{i_1},\dots, x_{i_t}]g)}{\partial\xx_3})
 =M_{\eta\pi_3(g)}\eta^e\pi_3^e(\frac{\partial[x_{i_1},\dots, x_{i_t}]}{\partial\xx_3})
 +M_{\eta\pi_3([x_{i_1},\dots, x_{i_t}])}\frac{\partial\gg}{\partial\xx_3}=0.$$
 Since elements of the form~$[x_{i_1},\dots, x_{i_t}]g$ ($t\geq 2$) linearly generate~$\mathsf{Ker}(\pi_3)$, the results follows.
 \end{proof}

 We conclude this section with a property of~$\eta^e(J(\varphi))$ for endomorphism~$\varphi$ of~$P_3$  such that~$\overline{\varphi}=\mathsf{id}$, the identity map of~$C_3$.

\begin{lemm}\label{etajvar-rough}
For all~$\varphi, \psi\in \mathsf{Aut}(P_3)$ satisfying~$\overline{\varphi}=\mathsf{id}$, we have
$$\eta^e(J(\psi\varphi\psi^{-1}))=
\begin{pmatrix}
  a_{11} & a_{12} &0\\
  a_{21}& a_{22} & 0\\
  a_{31} &a_{32} &1 \\
 \end{pmatrix}$$
 for some elements~$a_{p,q}$ $(1\leq p\leq 3, 1\leq q\leq 2)$ in~$P\{x_3\}^e$ such that
 $$ a_{11}a_{22}-a_{12}a_{21}=\eta^e\overline{\psi}^e(|J(\varphi)|).$$
Consequently, if~$\varphi$ is an elementary automorphism of the form~$\sigma(i,1,f)$ with~$f\in \mathsf{Ker}(\pi_3)$, then we have~$a_{11}a_{22}-a_{12}a_{21}=1$.
\end{lemm}
\begin{proof}
 Since~$\overline{\psi\varphi\psi^{-1}}=\overline{\varphi}=\mathsf{id}$, we may assume~$\psi\varphi\psi^{-1}=(x_1+f_1,x_2+f_2,x_3+f_3)$  and~$f_1,f_2,f_3\in \mathsf{Ker}(\pi_3)$.  By Lemma~\ref{par-ker}, we obtain the desired third column of~$\eta^e(J(\psi\varphi\psi^{-1}))$.
Then, by Lemma~\ref{app-chain-rule}, we have
$$a_{11}a_{22}-a_{12}a_{21}=|\eta^e(J(\psi\varphi\psi^{-1}))|=\eta^e(|(J(\psi\varphi\psi^{-1}))|)
=\eta^e \overline{\psi}^e(|J(\varphi)|).$$
For the consequence, it suffices to show~$|J(\varphi)|=1$ since~$\eta^e \overline{\psi}^e(1)=1$.
Suppose that~$J(\varphi)=(c_{pq})_{n\times n}$. Since~$f$ does not involve~$x_i$, by direct calculation, we have
\begin{equation}\label{matrix-varphi}
c_{pq}=
\begin{cases}
\delta_{p,q}, &   \mbox{if }  p\neq i , \\
1,                   &    \mbox{if }  p=q=i,  \\
\pi_3^e(\frac{\partial\ff}{\partial\xx_q}), &   \mbox{if } p=i, q\neq i.
\end{cases}
\end{equation}
The result follows immediately.
 \end{proof}

\section{The constructed Anick type automorphism is wild}\label{SS-wild}
Our aim in this section is to establish our main result: the Anick type automorphism~$\delta$ in~\eqref{Anick-auto} is wild. For an arbitrary associative algebra~$\mA$, we call matrices in~$M_{2\times 2}(\mA)$ of the form
$$E_{12}(a):=
\begin{pmatrix}
  1 & a \\
  0 & 1
\end{pmatrix}
$$
or
$$E_{21}(a):=
\begin{pmatrix}
  1 & 0 \\
  a & 1
\end{pmatrix}
$$
\emph{elementary $2\times 2$ matrices} or \emph{elementary  matrices}, where~$a\in \mathcal{A}$. We denote by $E_2(\mA)$ the subgroup of the general linear group~$GL_2(\mA)$ generated by all elementary $2\times 2$ matrices.
The following important and useful result is proved in~\cite{cohn-matrix}, see also~\cite{Umir07Anick}.

\begin{lemm}\cite{cohn-matrix}\label{matrix-not-ele}
Let~$k[x,y]$ be the polynomial algebra with variables~$x$ and~$y$. Then the matrix
$$\begin{pmatrix}
  1+xy & y^2\\
  -x^2 &  1-xy
\end{pmatrix} $$
does not lie in~$E_2(k[x,y])$.
Consequently, the transpose of the above matrix,
 $$\begin{pmatrix}
  1+xy & -x^2\\
 y^2  &  1-xy
\end{pmatrix}
 $$
also does not belong to~$E_2(k[x,y])$.
\end{lemm}
Let~$\varphi$ be an endomorphism of~$P_3$. Following the notation in~\cite{Umir07Anick},   we define
\begin{equation}\label{jvar2-def}
J_2(\varphi)=
\begin{pmatrix}
   \pi_3^e(\frac{\partial{\varphi(x_1)}}{\partial\xx_1})  & \pi_3^e(\frac{\partial{\varphi(x_1)}}{\partial\xx_2})\\
      \pi_3^e(\frac{\partial{\varphi(x_2)}}{\partial\xx_1})  & \pi_3^e(\frac{\partial{\varphi(x_2)}}{\partial\xx_2})
\end{pmatrix}.
\end{equation}

\begin{theorem}\label{thmj2}
Let~$\varphi\in \mathsf{Aut}(P_3)$ be an elementary automorphism of the form~$\sigma(i,1,f)$ with $f\in \mathsf{Ker}(\pi_3)$. Then for every  tame automorphism~$\psi\in \mathsf{TA}(P_3)$, we have $\eta^e(J_2(\psi\varphi\psi^{-1}))\in E_{2}(P\{x_3\}^e)$.
\end{theorem}

\begin{proof}
Since~$\psi$ is tame, we know that $\psi$ is a product of elementary automorphisms of the form~$\sigma(j,\alpha,g)$.  By Equations~\eqref{sigma1j} and~\eqref{sigma2j}, we may assume that~$\psi$ is a product of elementary automorphisms of the form
\begin{equation}\label{psiform}
  \sigma(1,\alpha, g) \mbox{ or }   \sigma(j,1,-x_1)  \mbox{ with } j\in\{2,3\}.
\end{equation}
So we may assume~$\psi=\psi_1\psi_2\dots \psi_{t}$, where each~$\psi_p$ is of the form in~\eqref{psiform}.

Now we use induction on~$t$ to prove the result.
For~$t=0$, we have~$\psi\varphi\psi^{-1}=\varphi=\sigma(i,1,f)$. By Lemma~\ref{par-ker}, we have~$\eta^e\pi_3^e(\frac{\partial{f}}{\partial\xx_3})=0$. Since~$f$ does not involve~$\xx_i$, we obtain~$\eta^e\pi_3^e(\frac{\partial{f}}{\partial\xx_i})=0$. There are several cases to consider. If~$i=1$, then by~\eqref{jvar2-def}, we have
 $$\eta^e(J_2(\varphi))=
 \begin{pmatrix}
  1 & \eta^e\pi_3^e(\frac{\partial\ff}{\partial\xx_2})\\
      0&1
\end{pmatrix}\in E_{2}(P\{x_3\}^e);$$
If~$i=2$, then we have
 $$\eta^e(J_2(\varphi))=
 \begin{pmatrix}
  1 &0\\
\eta^e\pi_3^e(\frac{\partial\ff}{\partial\xx_1})&1
\end{pmatrix}\in E_{2}(P\{x_3\}^e);$$
If~$i=3$, then we have
 $$\eta^e(J_2(\varphi))=
 \begin{pmatrix}
  1 &0\\
0&1
\end{pmatrix}\in E_{2}(P\{x_3\}^e).$$

Now we assume~$t>0$.  Denote~$\psi_2\dots \psi_t\varphi(\psi_2\dots \psi_t)^{-1}$ by~$\varphi_1$. Then by induction hypothesis and by Lemma~\ref{etajvar-rough}, we may assume
$$\eta^e(J(\varphi_1))
=\begin{pmatrix}
  a_{11} & a_{12} &0\\
  a_{21}& a_{22} & 0\\
  a_{31} &a_{32} &1 \\
\end{pmatrix}
=\begin{pmatrix}
 A_{2\times 2} & 0\\
 B_{1\times 2} & 1
\end{pmatrix}$$
such that~$A_{2\times 2}$ lies in~$ E_{2}(P\{x_3\}^e)$.
Denote by~$\widetilde{\psi_1}$ the automorphism of~$P\{x_3\}$ induced by~$\overline{\psi_1}$. Since the following diagram
$$\xymatrix{
  C_3 \ar[d]_{\eta} \ar[r]^{\overline{\psi_1}}
                &   C_3  \ar[d]^{\eta}  \\
 P\{x_3\} \ar[r]_{\widetilde{\psi_1}}
                &  P\{x_3\}            }
$$
 is commutative, we obtain that
$$
  \eta^{e}\overline{\psi_1}^e=(\eta\overline{\psi_1})^e=(\widetilde{\psi_1}\eta)^e
=\widetilde{\psi_1}^e\eta^e.
$$
Combining the above formula with Lemma~\ref{app-chain-rule}, we have
\begin{equation}\label{psieta}
\eta^e(J(\psi\varphi\psi^{-1}))=\eta^e(J(\psi_1\varphi_1\psi_1^{-1}))
 =\eta^e(J(\psi_1)^{-1})\widetilde{\psi_1}^e\eta^e(J(\varphi_1))\eta^e(J(\psi_1)).
\end{equation}
Assume that~$$\widetilde{\psi_1}^e\eta^e(J(\varphi_1))
=\begin{pmatrix}
\widetilde{\psi_1}^e(A_{2\times 2}) & 0\\
 \widetilde{\psi_1}^e(B_{1\times 2}) & 1
\end{pmatrix}
=\begin{pmatrix}
U_{2\times 2} & 0\\
V_{1\times 2} & 1
\end{pmatrix}.$$
Since~$\widetilde{\psi_1}^e$ maps elementary matrices over~$P\{x_3\}^e$
to elementary matrices over~$P\{x_3\}^e$, we deduce that~$U_{2\times 2}\in E_2(P\{x_3\}^e)$.

If~$\psi_1=\sigma(2,1,-x_1)$, then we have
$$
\eta^e(J(\psi_1))=
\begin{pmatrix}
1&0&0\\
-1&1&0\\
0&0&1
\end{pmatrix}
\mbox{ and }
\eta^e(J(\psi_1)^{-1})=
\begin{pmatrix}
1&0&0\\
1&1&0\\
0&0&1
\end{pmatrix}.$$
By~\eqref{psieta}, we deduce that
\begin{align*}
  \eta^e(J(\psi\varphi\psi^{-1}))
=&\begin{pmatrix}
E_{21}(1)&0 \\
0&1
\end{pmatrix}
\begin{pmatrix}
U_{2\times 2} & 0\\
V_{1\times 2} & 1
\end{pmatrix}
\begin{pmatrix}
E_{21}(-1)&0 \\
0&1
\end{pmatrix}&\\
=&\begin{pmatrix}
E_{21}(1)U_{2\times 2}E_{21}(-1)&0 \\
V_{1\times 2}E_{21}(-1)&1
\end{pmatrix}.&
\end{align*}
So we obtian
$$\eta^e(J_2(\psi\varphi\psi^{-1}))=E_{21}(1)U_{2\times 2}E_{21}(-1) \in E_2(P\{x_3\})^e.
$$

If~$\psi_1=\sigma(3,1,-x_1)$, then we have
$$
\eta^e(J(\psi_1))=
\begin{pmatrix}
1&0&0\\
0&1&0\\
-1&0&1
\end{pmatrix}
\mbox{ and }
\eta^e(J(\psi_1)^{-1})=
\begin{pmatrix}
1&0&0\\
0&1&0\\
1&0&1
\end{pmatrix}.$$
Denote the row matrix~$(-1,0)$ by~$W_{1\times 2}$. Then by~\eqref{psieta}, we deduce that
 \begin{align*}
  \eta^e(J(\psi\varphi\psi^{-1}))
=&\begin{pmatrix}
E_{2}&0 \\
-W_{1\times 2}&1
\end{pmatrix}
\begin{pmatrix}
U_{2\times 2} & 0\\
V_{1\times 2} & 1
\end{pmatrix}
\begin{pmatrix}
E_{2}&0 \\
W_{1\times 2}&1
\end{pmatrix}&\\
=&\begin{pmatrix}
 U_{2\times 2} &0 \\
 -W_{1\times 2}U_{2\times 2}+V_{1\times 2}+W_{1\times 2}&1
\end{pmatrix}.&
\end{align*}
So we obtian
$$\eta^e(J_2(\psi\varphi\psi^{-1}))=U_{2\times 2} \in E_2(P\{x_3\})^e.
$$
Finally, if~$\psi_1=\sigma(1,\alpha, g)$ with~$\alpha\neq 0$, say~$w_1=\eta^e\pi_3^e(\frac{\partial\gg}{\partial\xx_2})$ and~$w_2=\eta^e\pi_3^e(\frac{\partial\gg}{\partial\xx_3})$,
 then we have
$$
\eta^e(J(\psi_1))=
\begin{pmatrix}
\alpha  &   w_1   &     w_2\\
0          &    1     &     0\\
0          &    0     &      1
\end{pmatrix}
=\begin{pmatrix}
  W_{2\times 2}& G_{2\times 1}\\
  0                   &    1
\end{pmatrix}
$$
 and
 $$
\eta^e(J(\psi_1)^{-1})=
\begin{pmatrix}
\alpha^{-1}  &   -\alpha^{-1}w_1   &     -\alpha^{-1}w_2\\
0          &    1     &     0\\
0          &    0     &      1
\end{pmatrix}
=\begin{pmatrix}
  Y_{2\times 2}& Z_{2\times 1}\\
  0                   &    1
\end{pmatrix}.$$
By~\eqref{psieta} again, we deduce that
 \begin{align*}
  \eta^e(J(\psi\varphi\psi^{-1}))
=&\begin{pmatrix}
  Y_{2\times 2}& Z_{2\times 1}\\
  0                   &    1
\end{pmatrix}
\begin{pmatrix}
U_{2\times 2} & 0\\
V_{1\times 2} & 1
\end{pmatrix}
\begin{pmatrix}
  W_{2\times 2}& G_{2\times 1}\\
  0                   &    1
\end{pmatrix}
&\\
=&
\begin{pmatrix}
  Y_{2\times 2}U_{2\times 2}+Z_{2\times 1}V_{1\times 2} &   Z_{2\times 1}\\
  V_{1\times 2}                                                                 &    1
  \end{pmatrix}
  \begin{pmatrix}
  W_{2\times 2}& G_{2\times 1}\\
  0                   &    1
\end{pmatrix}
&\\
=&
\begin{pmatrix}
  Y_{2\times 2}U_{2\times 2}W_{2\times 2}+Z_{2\times 1}V_{1\times 2}W_{2\times 2}  \  &  \  Y_{2\times 2}U_{2\times 2}G_{2\times 1}+Z_{2\times 1}V_{1\times 2}G_{2\times 1}+Z_{2\times 1}\\
  V_{1\times 2}W_{2\times 2}                                                               &    V_{1\times 2}G_{2\times1}+1
  \end{pmatrix}
  &
\end{align*}
Then by Lemma~\ref{etajvar-rough}, we have
$$1=V_{1\times 2}G_{2\times1}+1.$$
and
$$Y_{2\times 2}U_{2\times 2}G_{2\times 1}+Z_{2\times 1}V_{1\times 2}G_{2\times 1}+Z_{2\times 1}=0.$$
Assume~$V_{1\times 2}=(v_1,v_2)$. Then we have
$$v_1w_2=V_{1\times 2}G_{2\times1}=0.$$
And thus
 $$Z_{2\times 1}=-Y_{2\times 2}U_{2\times 2}G_{2\times 1}.$$
 Since~$P\{x_3\}^e$ is commutative, we also have~$w_2v_1=0$.
Therefore, we deduce that
\begin{align*}
 & \eta^e(J_2(\psi\varphi\psi^{-1}))&\\
 =& Y_{2\times 2}U_{2\times 2}W_{2\times 2}+Z_{2\times 1}V_{1\times 2}W_{2\times 2}&\\
  =& Y_{2\times 2}U_{2\times 2}W_{2\times 2}-Y_{2\times 2}U_{2\times 2}G_{2\times 1}V_{1\times 2}W_{2\times 2}&\\
  =& Y_{2\times 2}U_{2\times 2}(E_2-G_{2\times 1}V_{1\times 2})W_{2\times 2}.& \\
  =& Y_{2\times 2}U_{2\times 2}E_{12}(-w_2v_2)W_{2\times 2}& \\
    =&
    \begin{pmatrix}
      \alpha^{-1} & 0\\
      0 & 1
    \end{pmatrix}
    E_{12}(-w_1)U_{2\times 2}E_{12}(-w_2v_2)E_{12}(w_1)
    \begin{pmatrix}
      \alpha & 0\\
      0& 1
    \end{pmatrix}.&
\end{align*}
Note that for all~$u \in P\{x_3\}^e$,  we have~\cite{cohn-matrix,Umir07Anick}
$$E_{12}(u) \begin{pmatrix}
      \alpha & 0\\
      0& 1
    \end{pmatrix}
    =\begin{pmatrix}
      \alpha & 0\\
      0& 1
    \end{pmatrix}E_{12}(\alpha^{-1}u)
    \mbox{ and }
    E_{21}(u) \begin{pmatrix}
      \alpha & 0\\
      0& 1
    \end{pmatrix}
    =\begin{pmatrix}
      \alpha & 0\\
      0& 1
    \end{pmatrix}E_{21}(\alpha\uu). $$
By the above reasoning, we obtain
$$ \eta^e(J_2(\psi\varphi\psi^{-1}))\in \begin{pmatrix}
      \alpha^{-1} & 0\\
      0 & 1
    \end{pmatrix} \begin{pmatrix}
      \alpha & 0\\
      0& 1
    \end{pmatrix}E_{2}(P\{x_3\}^e)
= E_{2}(P\{x_3\}^e).    $$
The proof is completed.
\end{proof}

\begin{coro}\label{tame-j2}
For~$\varphi\in \mathsf{TA}(P_3)$, if~$\overline{\varphi}=\mathsf{id}$, then~$\eta^e(J_2(\varphi))\in E_2(P\{x_3\}^e)$.
\end{coro}
\begin{proof}
  By Lemma~\ref{kernel-aut}, we may assume that
  $\varphi=\psi_1\varphi_1\psi_1^{-1}...\psi_t\varphi_t\psi_t^{-1}$ for some~$t\geq 1$, where each~$\varphi_j$ is an elementary automorphism of the form~\eqref{AutInKer} with~$n=3$.
  Define~$\theta_j=\psi_j\varphi_j\psi_j^{-1}$ for each~$1\leq j\leq t$.  Then~$\overline{\theta_j}=\mathsf{id}$. By Lemma~\ref{chain-rule} and by induction on~$t$, we have
 $$J(\varphi)=J(\theta_1\dots \theta_t)
 =\overline{\theta_1\dots \theta_{t-1}}^e(J(\theta_t))J(\theta_1\dots \theta_{t-1})
 =J(\theta_t)J(\theta_1\dots \theta_{t-1})
  =J(\theta_t)\dots J(\theta_1).$$
  By Lemma~\ref{etajvar-rough} and by Theorem~\ref{thmj2}, we may assume
  $$\eta^e(J(\theta_j))=
\begin{pmatrix}
  A_j &0\\
V_j & 1\\
 \end{pmatrix}$$
such that~$A_j\in E_2(P\{x_3\}^e)$. By induction on~$t$, it is straightforward to show that
$$\eta^e(J(\theta_1\dots \theta_t))
=\eta^e(J(\theta_t))\dots \eta^e(J(\theta_1))
=\begin{pmatrix}
  A_tA_{t-1}\dots A_1 & 0\\
  V& 1
\end{pmatrix}.
 $$
The result follows immediately.
\end{proof}
Now we are ready to prove our main result. Since~$\overline{\delta}\neq \mathsf{id}$, we need to construct a tame automorphism~$\psi$ such that~$\overline{\delta\psi}=\mathsf{id}$.
\begin{theorem}\label{thm-wild}
  The Anick type automorphism~$\delta$ in~\eqref{Anick-auto} is wild.
\end{theorem}
\begin{proof} Clearly, we have~$\overline{\delta}=(x_1, x_2+x_1x_3^2,x_3)$.
Define~$\psi=(x_1, x_2-x_1x_3^2,x_3)\in \mathsf{TA}(P_3)$. Then we have $\overline{\delta\psi}=\bar{\delta}\bar{\psi}=\mathsf{id}$.
Suppose that~$\delta$ is tame. Then by Corollary~\ref{tame-j2}, we have
$$\eta^e(J_2(\delta\psi))\in E_2(P\{x_3\}^e).$$
On the other hand, by Lemma~\ref{chain-rule}, we obtain
$$\eta^e(J(\delta\psi))=\eta^e\overline{\delta}^e(J(\psi))\eta^e(J(\delta))
=\begin{pmatrix}
  1&0&0\\
  -M_{x_3}^2 &1 &0\\
  0&0&1
\end{pmatrix}
\begin{pmatrix}
  1+M_{x_3}H_{x_3} & -H_{x_3}^2&0\\
  M_{x_3}^2 & 1-M_{x_3}H_{x_3} &0\\
  0&0&1
\end{pmatrix}
$$
So by Lemma~\ref{matrix-not-ele}, we deduce a contradiction that
$$\eta^e(J_2(\delta\psi))
=\begin{pmatrix}
  1&0\\
  -M_{x_3}^2 &1
\end{pmatrix}
\begin{pmatrix}
  1+M_{x_3}H_{x_3} & -H_{x_3}^2\\
  M_{x_3}^2 & 1-M_{x_3}H_{x_3} \\
\end{pmatrix}
\notin E_{2}(P\{x_3\}^e).
$$
Therefore, $\delta$ is wild.
\end{proof}

We conclude the article with the observation that, as for the case of a 3-generated free associative algebra~\cite{DY05},  the Anick type automorphism~$\delta$ is stably tame.  
Suppose that~$n=4$ and extend the Anick type automorphism~$\delta$ to
$$\varphi=(\delta(x_1),\delta(x_2),x_3,x_4).$$
Then we have~$\varphi=\varphi_8\dots\varphi_1$, where
$\varphi_1=\sigma(2,1,x_3x_4)$, $\varphi_2=\sigma(1,1,[x_3,x_4])$, $\varphi_3=\sigma(4,1,-[x_3,x_2])$, $\varphi_4=\sigma(4,1, x_1x_3)$, $\varphi_5=\sigma(2,1, -x_3x_4)$, $\varphi_6=\sigma(1, 1, -[x_3,x_4])$, $\varphi_7=\sigma(4,1,[x_3,x_2])$, $\varphi_8=\sigma(4, 1,-x_1x_3)$.

 For the convenience of the readers, we list the images below:

\ITEM1 $x_1\xrightarrow{\varphi_1} x_1\xrightarrow{\varphi_2} x_1+[x_3,x_4] \xrightarrow{\varphi_3}
 x_1+[x_3,x_4]-[x_3,[x_3,x_2]]\xrightarrow{\varphi_4}x_1+[x_3,x_4]+x_3[x_3,x_1]-[x_3,[x_3,x_2]]
 \xrightarrow{\varphi_5}x_1+[x_3,x_4]+x_3[x_3,x_1]-[x_3,[x_3,x_2]]+x_3[x_3,[x_3,x_4]
 \xrightarrow{\varphi_6} x_1+x_3[x_3,x_1]-[x_3,[x_3,x_2]]\xrightarrow{\varphi_8\varphi_7}x_1+x_3[x_3,x_1]-[x_3,[x_3,x_2]]
 =\delta(x_1)$;

 \ITEM2 $x_2\xrightarrow{\varphi_1} x_2+x_3x_4\xrightarrow{\varphi_2}x_2+x_3x_4
 \xrightarrow{\varphi_3} x_2+x_3x_4-x_3[x_3,x_2]\xrightarrow{\varphi_4}x_2+x_3x_4+x_1x_3^2-x_3[x_3,x_2]
 \xrightarrow{\varphi_5}
 x_2+x_1x_3^2-x_3[x_3,x_2]+x_3^2[x_3,x_4]
 \xrightarrow{\varphi_6}
 x_2+x_1x_3^2-x_3[x_3,x_2]
  \xrightarrow{\varphi_8\varphi_7} x_2+x_1x_3^2-x_3[x_3,x_2]=\delta(x_2)$;

  \ITEM3 $x_3\xrightarrow{\varphi}x_3 $;

  \ITEM4  $x_4\xrightarrow{\varphi_2\varphi_1} x_4\xrightarrow{\varphi_4\varphi_3}x_4+x_1x_3-[x_3,x_2]
  \xrightarrow{\varphi_5}x_4+x_1x_3-[x_3,x_2]+x_3[x_3,x_4]
  \xrightarrow{\varphi_6}x_4+x_1x_3-[x_3,x_2]
  \xrightarrow{\varphi_7}x_4+x_1x_3\xrightarrow{\varphi_8} x_4$.

\subsection*{Disclosure statement}
No potential conflict of interest was reported by the authors.

\subsection*{Acknowledgement} The authors would like to thank Yu Li, Yanhua Wang, Yongjun Xu, James J. Zhang, and Xiangui Zhao for their valuable discussions and suggestions.

\newcommand{\noopsort}[1]{}

\end{document}